\documentclass[a4paper,12pt]{amsart}
\usepackage{amssymb,amscd}
\usepackage[cp1251]{inputenc}

\usepackage{amsmath}
\usepackage{amsthm}
\usepackage{mathtext}
\usepackage{euscript}

\usepackage{graphicx}
\graphicspath{}
\DeclareGraphicsExtensions{.pdf,.png,.jpg,.jpeg}

\linethickness{0.5pt}

\usepackage[english]{babel}

\theoremstyle{plain}

\newtheorem{theo}{Theorem}[section]
\newtheorem{prop}[theo]{Proposition}
\newtheorem{lemm}[theo]{Lemma}
\newtheorem{coro}[theo]{Corollary}

\theoremstyle{definition}
\newtheorem{exam}[theo]{Example}
\newtheorem{constr}[theo]{Construction}
\newtheorem{defi}[theo]{Definition}

\newtheorem{prob}[theo]{Problem}

\theoremstyle{remark}
\newtheorem*{rema}{Remark}

\numberwithin{equation}{section}

\newcommand{\field}[1]{\mathbb{#1}}
\def\C{\field{C}}

\newcommand{\R}{\field{R}}

\DeclareMathOperator{\Tor}{Tor}
\DeclareMathOperator{\vc}{vc}
\DeclareMathOperator{\ko}{{\bf{k}}}
\DeclareMathOperator{\rk}{rk}
\DeclareMathOperator{\cc}{cc}
\DeclareMathOperator{\MF}{{\bf{MF}}}
\DeclareMathOperator{\conv}{conv}
\DeclareMathOperator{\sk}{sk}
\DeclareMathOperator{\reg}{reg}

\def\As{\mbox{\it As}}
\def\Pe{\mbox{\it Pe}}
\def\Cy{\mbox{\it Cy}}
\def\St{\mbox{\it St}}

\newcommand{\zp}{\mathcal Z_P}

\newcommand{\zk}{\mathcal Z_K}
\DeclareMathOperator{\Dj}{DJ}
\DeclareMathOperator{\U}{U}

\DeclareMathOperator{\K}{\mbox{\textit{K}}}

\begin{document}

\title[Topology of polyhedral products]{Topology of polyhedral products over simplicial multiwedges}

\author{Ivan Limonchenko}

\thanks{The author was supported by the General Financial Grant from the China Postdoctoral Science Foundation, grant no.~2016M601486.
}


\address{School of Mathematical Sciences, Fudan University, Handan Road 220, Yangpu, Shanghai, 200433, P.R. China}
\email{ilimonchenko@fudan.edu.cn}


\maketitle
\begin{abstract}
We prove that certain conditions on multigraded Betti numbers of a simplicial complex $K$ imply existence of a higher Massey product in cohomology of a moment-angle-complex $\zk$, which contains a unique element (a \emph{strictly defined} product). Using the simplicial multiwedge construction, we find a family $\mathcal{F}$ of polyhedral products being smooth closed manifolds such that for any $l, r\geq 2$ there exists an $l$-connected manifold $M\in\mathcal F$ with a nontrivial strictly defined $r$-fold Massey product in $H^{*}(M)$. As an application to homological algebra, we determine a wide class of triangulated spheres $K$ such that a nontrivial higher Massey product of any order may exist in Koszul homology of their Stanley--Reisner rings. As an application to rational homotopy theory, we establish a combinatorial criterion for a simple graph $\Gamma$ to provide a (rationally) formal generalized moment-angle manifold $\zp^{J}=(\b{D}^{2j_{i}},\b{S}^{2j_{i}-1})^{\partial P^*}$, $J=(j_{1},\ldots,j_m)$ over a graph-associahedron $P=P_{\Gamma}$ and compute all the diffeomorphism types of formal moment-angle manifolds over graph-associahedra. 
\end{abstract}

\section{Introduction}

A Massey product is a multivalued partially defined higher operation in cohomology of a space, which generalizes the notion of a cup product.
Since 1957, when the triple Massey product in the cohomology ring $H^{*}(X)$ of a topological space $X$ was introduced and applied~\cite{M-U} in order to prove the Jacobi identity for Whitehead products in $\pi_{*}(X)$, the higher order (ordinary) Massey operation, introduced in~\cite{Mass}, and its generalization, matric Massey product~\cite{M}, found numerous applications in geometry, topology, group theory, and other areas of research. 

Among the most remarkable of these applications, we can mention the one, which belongs to rational homotopy theory~\cite{sull77}: a nontrivial higher Massey product in (singular) cohomology ring $H^*(X;\mathbb{Q})$ is an obstruction to the (rational) formality of $X$. This allows one to prove that certain complex manifolds are non-K\"ahler due to the classical result of~\cite{D-G-M-S} on formality of K\"ahler manifolds, as well as to construct nonformal manifolds arising in symplectic geometry~\cite{Ba-Ta}, the theory of subspace arrangements~\cite{DS}, and toric topology~\cite{BaskM}.    

Since the pioneering work~\cite{da-ja91} appeared, toric topology has become a branch of geometry and topology that provides algebraic and equivariant topology, symplectic geometry, combinatorial commutative algebra, enumerative combinatorics, and polytope theory with a class of new objects such as small covers and quasitoric manifolds~\cite{da-ja91}, (generalized) moment-angle-complexes $\zk$~\cite{bu-pa00-2} and polyhedral products $({\bf{X}},{\bf{A}})^K$~\cite{BBCG10}, on which the general theory can be worked out explicitly and a number of remarkable properties of toric spaces with applications in other areas of research can be obtained. 

Moment-angle-complexes $\zk$ over simplicial complexes $K$ and moment-angle manifolds $\zp$ of simple polytopes $P$ have become the key objects of study in toric topology, and there have already been shown various connections of these spaces with objects arising in different areas of geometry and topology. Most of these developments have been recently summarized in the fundamental monograph~\cite{TT}. 

In particular, it was proved in~\cite{bu-pa00-2} that a moment-angle-complex $\zk$ is a deformation retract of a complement $\U(K)$ of the coordinate subspace arrangement in $\C^m$ corresponding to $K$, therefore, being homotopy equivalent to such a complement. Furthermore, a moment-angle manifold $\zp$ of an $n$-dimensional simple polytope $P$ with $m$ facets is a smooth $(m+n)$-dimensional 2-connected manifold embedded into $\C^m$ with a trivial normal bundle as a nondegenerate intersection of Hermitian quadrics. 

A $\mathbb{T}^m$-action on $\zp$ is induced from the standard (coordinatewise) action of $\mathbb{T}^m$ on $\C^m$, and the orbit space of this action is the convex simple polytope $P$ itself. Moreover, when the quasitoric manifold $M_P$ over $P$ exists (in particular, when $P$ has a Delzant realization, and one of the quasitoric manifolds over $P$ is then its nonsingular projective toric variety $V_P$), $\zp$ is a total space of a principal $\mathbb{T}^{m-n}$-bundle over $M_P$ and a composition of the latter bundle with a projection of $M_P$ onto the orbit space of the $\mathbb{T}^n$-action on it coincides with the projection of $\zp$ onto the orbit space of the big torus $\mathbb{T}^m$ action mentioned above.

In a series of more recent works~\cite{GT,G-T13,IK,G-P-T-W,BG} the topological properties of moment-angle-complexes $\zk$ were related to the algebraic properties of their Stanley--Reisner algebras $\ko[K]$ (over a field, or a ring with unit $\ko$). One of such connections is provided by the notion of a Golod ring, which appeared firstly in homology theory of local rings in 1960s, see~\cite{Go}, and has recently acquired the following meaning in toric topology (see Theorem~\ref{zkcoh}): a face ring $\ko[K]$ is called \emph{Golod} (over $\ko$) if multiplication (cup product) and all higher Massey products in $H^{*}(\zk;\ko)$ are vanishing. An up-to-date comprehensive survey on homotopy theory of polyhedral products with applications to the problem of finding conditions on $K$ that imply Golodness of $\ko[K]$ (over $\ko$) can be found in~\cite{GT2}.

It turned out that most of the classes of Golod complexes $K$ that arise in the way described above produce $\zk$ homotopy equivalent to a suspension space (or, even to a wedge of spheres). Furthermore, there have already been established several classes of $K$ (or $P$) for which $\zk$ (resp. $\zp$) is homeomorphic to a connected sum of products of spheres thus also being formal spaces~\cite{bo-me06,G-LdM,G-P-T-W}. Later on, several rational models for toric spaces, among which are moment-angle-complexes, Davis-Januszkiewicz spaces, and quasitoric manifolds, were constructed~\cite{p-r-v04,DS,P-R} and, in particular, it was proved that all quasitoric manifolds are formal~\cite{P-R}.

On the other hand, it was first shown in~\cite{BaskM} that there exist polyhedral products having nontrivial higher Massey products in cohomology and, therefore, being nonformal. Combinatorial criterion for a graph of a simplicial complex $K$ to provide a moment-angle complex $\zk$ with a nontrivial triple Massey product of 3-dimensional classes in $H^*(\zk)$ was proved in~\cite{DS}. Using this result, we prove a necessary and sufficient condition in terms of combinatorics of a simple graph $\Gamma$ to provide a moment-angle manifold $\zp$ over a graph-associahedron $P=P_{\Gamma}$ with a nontrivial triple Massey product in $H^*(\zp)$, see Lemma~\ref{MasseyGA}. Recently, it was shown in~\cite{Zh} that there exists a nontrivial triple Massey product in $H^*(\zp)$ for any Pogorelov polytope $P$ (the latter class of 3-dimensional flag simple polytopes contains, in particular, all fullerenes). First examples of polyhedral products with nontrivial $n$-fold Massey products in cohomology for any $n\geq 2$ were constructed in~\cite{L2,L1}.

The simplicial multiwedge operation (or, $J$-construction) was introduced in the framework of toric topology in~\cite{BBCG15} and was then used to prove that generalized moment-angle-complexes of certain types are homeomorphic to (real) moment-angle-complexes, see Theorem~\ref{homeoBBCG}. Therefore, a natural problem arises: to determine a class of generalized moment-angle-complexes that are (rationally) formal, see Problem~\ref{probFormal}, and to construct a nontrivial $n$-fold Massey product for any $n\geq 2$ with a possibly small indeterminacy in (rational) cohomology of generalized moment-angle-complexes. In this paper we are going to deal with the above mentioned problem.

The structure of the paper is as follows. In Section 2 we give a survey of the main definitions, constructions, and results concerning simplicial complexes, $J$-construction, and polyhedral products that we need in order to state and prove our main results. 

In Section 3 we show that certain conditions on the induced subcomplexes in a simplicial complex $K$, more precisely, vanishing of particular multigraded Betti numbers of $K$ imply existence of a strictly defined higher Massey product in $H^*(\zk)$, see Lemma~\ref{strictMassey}. Then we apply this result and the simplicial multiwedge construction in order to find a family of generalized moment-angle manifolds such that for any $l, r\geq 2$ this family contains an $l$-connected manifold $M$ with a nontrivial strictly defined $r$-fold Massey product in $H^{*}(M)$, see Theorem~\ref{mainMassey}. It turns out that the underlying polytopes are 2-truncated cubes, therefore, they are Delzant and thus any nonformal toric space $M$ from our family is a total space of a principal toric bundle over a nonsingular projective toric variety. The latter one is a formal manifold and is equivariantly symplectomorphic to a compact connected symplectic manifold with an effective Hamiltonian action of a half-dimensional compact torus~\cite{delz88}. 

Section 4 is devoted to applications of our main results on higher Massey products to homological algebra and rational homotopy theory. As an application to homological algebra, we determine a wide class of triangulated spheres $K$ such that a nontrivial higher Massey product of any order may exist in Koszul homology of their Stanley--Reisner rings, see Theorem~\ref{coromainMassey} and Proposition~\ref{MasseyDegrees}. As an application to rational homotopy theory, we establish a combinatorial criterion for a simple graph $\Gamma$ to provide a (rationally) formal generalized moment-angle manifold $\zp^{J}=(\b{D}^{2j_{i}},\b{S}^{2j_{i}-1})^{\partial P^*}$, $J=(j_{1},\ldots,j_m)$ over a graph-associahedron $P=P_{\Gamma}$ and compute all the diffeomorphism types of formal moment-angle manifolds over graph-associahedra, see Theorem~\ref{mainFormal}. 

We are grateful to Anthony Bahri, Martin Bendersky, and Taras Panov for drawing our attention to studying nontrivial higher Massey products in cohomology of generalized moment-angle-complexes. We also thank Lukas Katth\"an for fruitful discussions on the combinatorial commutative algebra of face rings and Massey products in their Tor-algebras.

\section{Moment-angle-complexes and simplicial multiwedges}

We start with the following basic definition.

\begin{defi}
An \emph{(abstract) simplicial complex} $K$ on a vertex set $[m]=\{1,\ldots,m\}$ is a subset of $2^{[m]}$ such that if $\sigma\in K$ and $\tau\subseteq\sigma$, then $\tau\in K$.

The elements of $K$ are called its \emph{simplices}, and the maximal dimension of a simplex $\sigma\in K$ is the \emph{dimension} of $K$ and is denoted by $\dim K=n-1$, where $\sigma=\Delta^{n-1}$. Finally, for any vertex set $I\subset [m]$ a subset of $K$ that equals the intersection $K\cap 2^{I}$ is obviously a simplicial complex itself (on the vertex set $I$). It is called \emph{an induced subcomplex} in $K$ (on $I$) and denoted by $K_{I}$.
\end{defi}

In what follows we assume that there are no \emph{ghost vertices} in $K$, that is,  $\{i\}\in K$ for every $1\leq i\leq m$. Note that a simplicial complex $K$ is a poset, the natural ordering is by inclusion. Thus, $K$ (and any of its induced subcomplexes $K_{I}, I\subset [m]$) is defined by its maximal (w.r.t. inclusion) simplices and $\dim K$ is equal to the dimension of one of them. If all maximal simplices of $K$ have the same dimension, then $K$ is called a \emph{pure} simplicial complex. 

\begin{exam}
In what follows we denote by $P$ an $n$-dimensional convex simple polytope with $m$ \emph{facets} $F_{1},\ldots,F_{m}$, i.e. faces of codimension 1. In this paper we are interested only in its face poset structure (i.e. its combinatorial equivalence class), not in its particular embedding in the ambient Euclidean space $\R^n$. Consider the nerve of the (closed) covering of $\partial P$ by all the facets $F_{i}, 1\leq i\leq m$. The resulting simplicial complex on the vertex set $[m]$ will be called the \emph{nerve complex} of the simple polytope $P$ and denoted by $K_P$. Note that the geometric realization of $K_P$ is the boundary of the simplicial polytope $P^*$ combinatorially dual to $P$. It is obviously a pure simplicial complex and its dimension is equal to $n-1$.  
\end{exam}

One can see easily that a simplicial complex $K$ can either be defined by all of its maximal simplices, or, alternatively, by all of its minimal non-faces, that is, elements $I\in 2^{[m]}\backslash K$, minimal w.r.t. inclusion. By the above definition, the latter is equivalent to the property that $I\notin K$, but any of its proper subsets is a simplex in $K$. We denote the set of minimal non-faces of $K$ by $\MF(K)$. Therefore, $I\in\MF(K)$ if and only if $K_I$ is a boundary of a simplex $\Delta_I$ on the vertex set $I$. Using this set, we can easily define the following combinatorial operation on the sets of simplicial complexes and simple polytopes, which was brought to toric topology by Bahri, Bendersky, Cohen, and Gitler~\cite{BBCG15}. 

\begin{constr}\label{simpmultwedge}
Let $J=(j_{1},\ldots,j_{m})$ be an $m$-tuple of positive integers. Consider the following vertex set:
$$
m(J)=\{11,\ldots,1j_{1},\ldots,m1,\ldots,mj_{m}\}.
$$
To define \emph{a simplicial multiwedge, or a $J$-construction} of $K$, which is a simplicial complex $K(J)$ on $m(J)$, we say that the set $\MF(K(J))\subset 2^{[m(J)]}$ consists of the subsets of $m(J)$ of the type 
$$
I(J)=\{i_{1}1,\ldots,i_{1}j_{i_1},\ldots,i_{k}1,\ldots,i_{k}j_{i_k}\},
$$
where $I=\{i_{1},\ldots,i_{k}\}\in\MF(K)$. Note that if $J=(1,\ldots,1)$, then $K(J)=K$. 
\end{constr}

\begin{exam}
Suppose $K=K_P$ is a polytopal sphere on $m$ vertices, that is, a nerve complex of a simple $n$-dimensional polytope with $m$ facets. Due to~\cite[Theorem 2.4]{BBCG15} and~\cite[Proposition 2.2]{Ch-P} its simplicial multiwedge $K(J)$ is always a polytopal sphere and thus a nerve complex of a simple polytope $Q=P(J)$, and so $K_{P(J)}=K_{P}(J)$. If we denote by $d(J)=j_{1}+\ldots+j_{m}$ then it is easy to see that $m(P(J))=d(J)$ and $n(P(J))=d(J)+n-m$. Therefore, $m(P(J))-n(P(J))=m(P)-n(P)$ remains the same after performing a $J$-construction. 
\end{exam}

Note that when $J=(2,\ldots,2)$ one gets the so called \emph{doubling} of the simple polytope $P$; for the geometric description of the $J$-construction on polytopes and its applications in toric topology see also~\cite{G-LdM}.  

In what follows we shall denote by $\ko$ a field of zero characteristic, or the ring of integers. Let $\ko[v_{1},\ldots,v_{m}]$ be a graded polynomial algebra on $m$ variables, $\deg(v_{i})=2$. 

\begin{defi}
A \emph{Stanley--Reisner ring}, or a \emph{face ring} of $K$ (over $\ko$) is the quotient ring
$$
   \ko[K]=\ko[v_{1},\ldots,v_{m}]/\mathcal I_{SR}(K),
$$
where $\mathcal I_{SR}(K)$ is the ideal generated by square free
monomials $v_{i_{1}}\cdots{v_{i_{k}}}$ such that $\{i_{1},\ldots,i_{k}\}\in\MF(K)$. The monomial ideal
$\mathcal I_{SR}(K)$ is called the \emph{Stanley--Reisner ideal} of~$K$.
Then $\ko[K]$ has a natural structure of a $\ko$-algebra and a module over the polynomial algebra $\ko[v_{1},\ldots,{v_{m}}]$ via the quotient projection. 
\end{defi}

Note that $\mathcal{I}_{SR}(K)=(0)$ if and only if $K=\Delta_{[m]}$.
By a result of Bruns and Gubeladze~\cite{br-gu96} two simplicial complexes $K_1$ and $K_2$ are combinatorially equivalent if and only if their Stanley--Reisner algebras are isomorphic. Thus, $\ko[K]$ is a complete combinatorial invariant of a simplicial complex $K$.

\begin{defi}
A simplicial complex $K$ is called \emph{flag} if it coincides with the \emph{clique complex} $\Delta(\Gamma)$ of its 1-dimensional skeleton $\Gamma=\sk^1(K)$, that is, for any subset of vertices $I\subset [m]$ of $K$, which are pairwisely connected by edges in $\Gamma$, its induced subcomplex $K_{I}=\Delta_{I}$. A simple polytope $P$ is called \emph{flag} if its nerve complex $K_P$ is flag.

A simplicial complex $K$ is called $q$-\emph{connected} ($q\geq 1$) if $K_{I}=\Delta_{I}$ for any subset of vertices $I\subset [m]$ of $K$ with $|I|=q$ elements. Thus, any simplicial complex $K$ is 1-connected; $K$ is $q$-connected implies $K$ is $l$-connected for any $1\leq l\leq q$.  
\end{defi}

\begin{rema}
(1) $K$ is flag if and only if for any $I\in\MF(K)$ one has $|I|=2$, i.e. $\mathcal I_{SR}(K)$ is generated by monomials of degree 4. (2) $K$ is $q$-connected if and only if for any $I\in\MF(K)$ one has $|I|\geq q+1$.
\end{rema}

Suppose $({\bf{X}},{\bf{A}})=\{(X_i,A_i)\}_{i=1}^{m}$ is a set of topological pairs. A particular case of the following construction appeared firstly in the work of Buchstaber and Panov~\cite{bu-pa00-2} and then was studied intensively and generalized in a series of more recent works, among which are~\cite{BBCG10,GT,IK}. 

\begin{defi}(\cite{BBCG10})
A {\emph{polyhedral product}} over a simplicial complex $K$ on the vertex set $[m]$ is a topological space
$$
({\bf{X}},{\bf{A}})^K=\bigcup\limits_{I\in K}({\bf{X}},{\bf{A}})^I,
$$
where $({\bf{X}},{\bf{A}})^I=\prod\limits_{i=1}^{m} Y_{i}$ for $Y_{i}=X_{i}$, if $i\in I$, and $Y_{i}=A_{i}$, if $i\notin I$.
\end{defi}

\begin{exam}
Suppose $X_{i}=X$ and $A_{i}=A$ for all $1\leq i\leq m$. Then the following spaces are particular cases of the polyhedral product construction $({\bf{X}},{\bf{A}})^{K}=(X,A)^K$.
\begin{itemize}
\item[(1)] The {\emph{moment-angle-complex}} $\zk=(\mathbb{D}^2,\mathbb{S}^1)^K$;
\item[(2)] The {\emph{real moment-angle-complex}} $\mathcal R_K=(\mathbb{D}^1,\mathbb{S}^0)^K$;
\item[(3)] If $K=K_P$, then the {\emph{moment-angle manifold}} $\zp=\zk$ is a 2-connected smooth closed $(m+n)$-dimensional manifold for any simple $n$-dimensional polytope $P$ with $m$ facets, see~\cite{bu-pa00-2};
\item[(3)] The \emph{Davis--Januszkiewicz space} $\Dj(K)=(\mathbb{C}P^{\infty},*)^K$;
\item[(4)] A complement of the coordinate subspace arrangement in $\C^m$  
$$
\U(K)=\C^m\backslash\cup_{I\in\MF(K)}\{(z_{1},\ldots,z_{m})|\,z_{i_{1}}=\ldots=z_{i_k}=0\}=(\mathbb{C},\mathbb{C}^{*})^K,
$$
determined by a simplicial complex $K$ ($I=\{i_{1},\ldots,i_{k}\}$);
\item[(5)] A cubical subcomplex $\cc(K)=(I^1,1)^K$ in $I^{m}=[0,1]^m$ which is PL homeomorphic to a cone over a barycentric subdivision of $K$.
\end{itemize} 
\end{exam}

\begin{rema}
The following properties of the polyhedral products defined above hold.
\begin{itemize}
\item $\zk$ is homeomorphic to $\mathcal R_{K(2,\ldots,2)}$;  
\item If $K$ is $q$-connected, then $\zk$ is a $2q$-connected CW complex;
\item There is a $\mathbb{T}^m$-equivariant deformation retraction: 
$$
\zk\hookrightarrow\U(K)\rightarrow\zk,
$$
in particular, $\U(K)$ and $\zk$ have the same homotopy type and $H^*(\U(K))\cong H^*(\zk)$;
\item One has a commutative diagram
$$\begin{CD}
  \zk @>>>(\mathbb{D}^2)^m\\
  @VVrV\hspace{-0.2em} @VV\rho V @.\\
  \cc(K) @>i_c>> I^m
\end{CD}\eqno 
$$
where $i_{c}:\,\cc(K)\hookrightarrow I^{m}=(I^1,I^1)^{[m]}$ is an embedding of a cubical subcomplex, induced by the inclusion of pairs: $(I^1,1)\subset (I^1,I^1)$, and the maps $r$ and $\rho$ are projections onto the orbit spaces of $\mathbb{T}^m$-action induced by the coordinatewise action of $\mathbb{T}^m$ on the unitary complex polydisk $(\mathbb{D}^2)^m$ in $\C^m$;
\item One has a homotopy fibration
$$\begin{CD}
  X @>>> ET^m\\
  @VVV\hspace{-0.2em} @V\pi VV @.\\
  \Dj(K) @>p>> BT^{m}
\end{CD}\eqno 
$$
where $\pi$ is the universal $\mathbb{T}^m$-bundle and the map $p:\,\Dj(K)\rightarrow BT^{m}$ is induced by the inclusion of pairs $(\mathbb{C}P^{\infty},*)\subset (\mathbb{C}P^{\infty},\mathbb{C}P^{\infty})$. 
Moreover, its homotopy fiber $X\simeq\zk$ and $\Dj(K)$ is homotopy equivalent to the Borel construction $ET^m\times_{\mathbb{T}^m}\zk$. Therefore, for the equivariant cohomology of $\zk$ we have:
$$
H^{*}_{\mathbb{T}^m}(\zk)=H^{*}(\Dj(K))\cong\mathbb{Z}[K].
$$
The details can be found in~\cite[Chapter 4, Chapter 8]{TT}.
\end{itemize}
\end{rema}

Buchstaber and Panov~\cite{bu-pa00-2} analysed the homotopy fibration above and proved that the corresponding Eilenberg-Moore spectral sequence over rationals, which converges to $H^{*}(\zk,\mathbb{Q})$, degenerates in the $E_2$ term. Furthermore, together with Baskakov they proved that the following result on the cohomology algebra of $\zk$ holds with coefficients in any ring with unit $\ko$. 

\begin{theo}[{\cite[Theorem 4.5.4]{TT}}]\label{zkcoh}
Cohomology algebra of a moment-angle-complex $\zk$ is given
by the isomorphisms
\[
\begin{aligned}
  H^{*,*}(\mathcal Z_K;\ko)&\cong\Tor_{\ko[v_1,\ldots,v_m]}^{*,*}(\ko[K],\ko)\\
  &\cong H^*\bigl[R(K),d\bigr]\\
  &\cong \bigoplus\limits_{I\subset [m]}\widetilde{H}^{*}(K_{I};\ko),
\end{aligned}
\]
where bigrading and differential in the cohomology of the differential
(bi)graded algebra are defined by
\[
  \mathop{\mathrm{bideg}} u_i=(-1,2),\;\mathop{\mathrm{bideg}} v_i=(0,2);\quad
  du_i=v_i,\;dv_i=0
\]
and $R^*(K)=\Lambda[u_{1},\ldots,u_{m}]\otimes\ko[K]/(v_{i}^{2}=u_{i}v_{i}=0,1\leq i\leq m)$.
Here, $\widetilde{H}^*(K_{I})$ (we drop $\ko$ from the notation in what follows) denotes the reduced simplicial cohomology of $K_{I}$. The last isomorphism above is the sum of isomorphisms 
$$
H^{p}(\zk)\cong\sum\limits_{I\subset [m]}\widetilde{H}^{p-|I|-1}(K_{I}).
$$
To determine the product of two cohomology classes $\alpha=[a]\in\tilde{H}^{p}(K_{I_1})$ and $\beta=[b]\in\tilde{H}^{q}(K_{I_2})$ define the natural inclusion of sets $i:\,K_{I_{1}\sqcup I_{2}}\rightarrow K_{I_1}*K_{I_2}$ and the canonical isomorphism of cochain modules:
$$
j:\,\tilde{C}^{p}(K_{I_1})\otimes\tilde{C}^{q}(K_{I_2})\rightarrow\tilde{C}^{p+q+1}(K_{I_{1}}*K_{I_2}).
$$
Then, the product of $\alpha$ and $\beta$ is given by:
$$
\alpha\cdot\beta=\begin{cases}
0,&\text{if $I_{1}\cap I_{2}\neq\varnothing$;}\\
i^{*}[j(a\otimes b)]\in\tilde{H}^{p+q+1}(K_{I_{1}\sqcup I_{2}}),&\text{if $I_{1}\cap I_{2}=\varnothing$.}
\end{cases}
$$
\end{theo}

The additive structure of the Tor-algebra that appears in the theorem above can be computed by using either the Koszul minimal free resolution for $\ko$ viewed as a $\ko[m]=\ko[v_1,\ldots,v_m]$-module, or the Taylor free resolution for $\ko[K]$ viewed as a $\ko[m]$-module. In general, the latter one is not minimal, however, it is sometimes more useful in combinatorial proofs. In fact, the $\ko$-module structure of $\Tor_{\ko[v_1,\ldots,v_m]}^{*,*}(\ko[K],\ko)$ is determined by all the reduced cohomology groups of all the induced subcomplexes in $K$ (including $\varnothing$ and $K$ itself). More precisely, the following result of Hochster holds.

\begin{theo}[{\cite{Hoch}}]\label{hoch}
For any simplicial complex $K$ on $m$ vertices:
$$
\Tor^{-i,2j}_{\ko[v_{1},\ldots,v_{m}]}(\ko[K],\ko)\cong\bigoplus\limits_{J\subset [m],\,|J|=j}\widetilde{H}^{j-i-1}(K_{J};\ko).
$$
\end{theo}

The ranks of the bigraded components of the Tor-algebra 
$$
\beta^{-i,2j}(\ko[K])=\rk_{\ko}\Tor^{-i,2j}_{\ko[v_{1},\ldots,v_{m}]}(\ko[K],\ko)
$$ 
are called the \emph{bigraded (algebraic) Betti numbers} of $\ko[K]$, or $K$ when $\ko$ is fixed. Note that by Theorem~\ref{zkcoh} they determine the (topological) Betti numbers of $\zk$.

Consider the polytopal case $K=K_P$. Let us denote by $P_J=\cup_{j\in J}\,F_{j}\subset\partial P$ the union of the corresponding facets of $P$ and by $\cc(P_J)$ -- the number of connected components of $P_J$. It can be seen that the following equality arises due to Theorem~\ref{hoch} when $j=i+1$: 
$$
\beta^{-i,2(i+1)}(P)=\sum\limits_{J\subset [m],|J|=i+1}(\cc(P_J)-1).
$$
 
In Section 3 we shall use a refinement of Theorem~\ref{zkcoh}. Namely, the differential graded algebra $R(K)$ acquires multigrading and the next statement holds.

\begin{theo}[{\cite[Construction 3.2.8, Theorem 3.2.9]{TT}}]\label{mgrad}
For any simplicial complex $K$ on $m$ vertices we have:
$$
\Tor^{-i,2J}_{\ko[v_{1},\ldots,v_{m}]}(\ko[K],\ko)\cong\widetilde{H}^{|J|-i-1}(K_{J};\ko),
$$
where $J\subset [m]$ and $\Tor^{-i,2{\bf{a}}}_{\ko[v_{1},\ldots,v_{m}]}(\ko[K],\ko)=0$, if ${\bf{a}}$ is not a $(0,1)$-vector. Moreover, 
$$
\Tor^{-i,2{\bf{a}}}_{\ko[v_{1},\ldots,v_{m}]}(\ko[K],\ko)\cong H^{-i,2{\bf{a}}}[R(K),d].
$$
\end{theo}

Now consider the real case. By~\cite[Theorem 8.9]{BP02} one has additive isomorphisms:
$$
H^{p}(\mathcal R_K;\ko)\cong\sum\limits_{J\subset [m]} H^{p-1}(K_J;\ko),
$$ 

The multiplicative structure was given firstly by Cai~\cite{C} and for generalized moment-angle-complexes by Bahri, Bendersky, Cohen, and Gitler~\cite{BBCG15}, which is equivalent to the result of~\cite{C} in the case of a real moment-angle-complex. Consider the differential graded algebra $r^*(K)$ which is a quotient algebra of a free graded algebra on $2m$ variables $u_{i},t_{j}$, where $\deg(u_{i})=1,\deg(t_{j})=0$, over the Stanley-Reisner ideal of $K$ in variables $\{u_i\}$ and the following commutation relations:
$$
u_{i}t_{i}=u_{i}, t_{i}u_{i}=0, u_{i}t_{j}=t_{j}u_{i}, t_{i}t_{i}=t_{i}, t_{i}t_{j}=t_{j}t_{i}, u_{i}u_{i}=0, u_{i}u_{j}=-u_{j}u_{i}.
$$ 

\begin{theo}[{\cite{C}}]
The following graded ring isomorphism holds:
$$
H^{*}(\mathcal R_K)\cong H^{*}[r(K),d],
$$ 
where $d(t_i)=u_i$ and $d(u_i)=0$.
\end{theo}

Moreover, if $K(J)$ is a simplicial multiwedge over $K$, then one has the next result on cohomology algebra of a generalized moment-angle-complex $\mathcal R_{K}^{J}=(\b{D}^{j_{i}},\b{S}^{j_{i}-1})^K$ ($1\leq i\leq m$), see~\cite[Theorem 5.1]{C} and~\cite[Proposition 6.2]{BBCG15} (for even $j_i$):
\begin{theo}
The following isomorphism of graded algebras holds:
$$
H^{*}(\mathcal R_{K}^{J})\cong H^{*}(\mathcal R_{K(J)}).
$$
\end{theo}

Note, that the algebras $R(K)$ and $r(K)$ are finitely generated $\ko$-modules, opposite to the case of the Koszul algebra $\Lambda[u_{1},\ldots,u_{m}]\otimes\ko[K]$, which has countably many additive generators. By the explicit constructions of multiplication in $H^*(\zk)$ and $H^*(\mathcal R_K)$, one gets that $R(K)$ and $r(K)$ are quasi-isomorphic to the cellular cochain algebras of $\zk$ and $\mathcal R_K$, respectively.

It can be seen that in the case $J=(2,\ldots,2)$ the algebra $r(K(J))$ is isomorphic to the algebra $R(K)$, introduced in Theorem~\ref{zkcoh}. Moreover, the real moment-angle-complex $\mathcal R_{K(2,\ldots,2)}$ is homeomorphic to the moment-angle-complex $\zk$ of the initial simplicial complex $K$ and, more generally, if $(\mathbb{D},\mathbb{S})=(\b{D}^{j_{i}},\b{S}^{j_{i}-1})$ ($1\leq i\leq m$), then by~\cite{BBCG15}:

\begin{theo}\label{homeoBBCG}
The generalized moment-angle-complex $(\mathbb{D}, \mathbb{S})^K$ is homeomorphic to the real moment-angle-complex $\mathcal R_{K(J)}$.
\end{theo}

\section{Main results}

In this section we are going to introduce our main results concerning (ordinary) higher Massey product structure in cohomology of certain polyhedral products. 

We start with a definition of a Massey product in cohomology of a differential graded algebra. First, we need a notion of a defining system for an ordered set of cohomology classes. Our presentation follows~\cite{Kr} and~\cite[Appendix $\Gamma$]{BP04}, in which all the properties of the Massey operation, necessary in what follows, can be found.

Suppose $(A,d)$ is a dga, $\alpha_{i}=[a_{i}]\in H^{*}[A,d]$ and $a_{i}\in A^{n_{i}}$ for $1\leq i\leq k$.
Then a \emph{defining system} for $(\alpha_{1},\ldots,\alpha_{k})$ is a $(k+1)\times (k+1)$-matrix $C$, such that the following conditions hold:
\begin{itemize}
\item[{(1)}] $c_{i,j}=0$, if $i\geq j$,
\item[{(2)}] $c_{i,i+1}=a_{i}$,
\item[{(3)}] $a\cdot E_{1,k+1}=dC-\bar{C}\cdot C$ for some $a=a(C)\in A$, where $\bar{c}_{i,j}=(-1)^{\deg(c_{i,j})}\cdot c_{i,j}$ and $E_{1,k+1}$ is a $(k+1)\times(k+1)$-matrix with '1' in the position $(1,k+1)$ and with all other entries being zero.
\end{itemize} 

A straightforward calculation shows that $d(a)=0$ and $a\in A^{m}$, $m=n_{1}+\ldots+n_{k}-k+2$. Thus, $a=a(C)$ is a cocycle for any defining system $C$, and its cohomology class $\alpha=[a]$ is defined.

\begin{defi}\label{defMassey}
A \emph{$k$-fold Massey product} $\langle\alpha_{1},\ldots,\alpha_{k}\rangle$ is said to be \emph{defined}, if there exists a defining system $C$ for it.
If so, this Massey product consists of all $\alpha=[a(C)]$, where $C$ is a defining system. A defined Massey product $\langle\alpha_{1},\ldots,\alpha_{k}\rangle$ is called 
\begin{itemize}
\item \emph{trivial} (or, \emph{vanishing}), if $[a(C)]=0$ for some $C$;
\item \emph{decomposable}, if $[a(C)]\in H^{+}(A)\cdot H^{+}(A)$ for some $C$;
\item \emph{strictly defined}, if $\langle\alpha_{1},\ldots,\alpha_{k}\rangle=\{[a(C)]\}$ for some $C$.
\end{itemize}
\end{defi}

\begin{rema}
Due to~\cite[Theorem 3]{Kr}, the set of cohomology classes $\langle\alpha_{1},\ldots,\alpha_{k}\rangle$ depends on the cohomology classes $\{\alpha_{1},\ldots,\alpha_{k}\}$, rather than on the particular representatives $\{a_{1},\ldots,a_{k}\}$. Therefore, if $[a_{i}]=0$ for some $1\leq i\leq k$ and $\langle\alpha_{1},\ldots,\alpha_{k}\rangle$ is defined, then $0\in\langle\alpha_{1},\ldots,\alpha_{k}\rangle$, thus the $k$-fold Massey product is trivial.
\end{rema}

\begin{exam}\label{examMassey}
Let us give all the relations on the elements of a defining system $C$ for a 2-, 3-, and 4-fold defined Massey product.

\begin{itemize}
\item[(1)] Suppose $k=2$.\\
If $\langle\alpha_{1},\alpha_{2}\rangle$ is defined, then we have:
$$
a=d(c_{1,3})-\bar{a}_1\cdot a_{2};
$$
Thus, our definition gives a usual cup product (up to sign) in (singular) cohomology of a space.

\item[(2)] Suppose $k=3$.\\
If $\langle\alpha_{1},\alpha_{2},\alpha_{3}\rangle$ is defined, then we have:
$$
a=d(c_{1,4})-\bar{a}_1\cdot c_{2,4}-\bar{c}_{1,3}\cdot a_{3},
$$
\[
\begin{aligned}
d(c_{1,3})&=&\bar{a}_1\cdot a_{2},\\
d(c_{2,4})&=&\bar{a}_2\cdot a_{3};    
\end{aligned}
\]

\item[(3)] Suppose $k=4$.\\
If $\langle\alpha_{1},\alpha_{2},\alpha_{3},\alpha_{4}\rangle$ is defined, then we have:
$$
a=d(c_{1,5})-\bar{a}_1\cdot c_{2,5}-\bar{c}_{1,3}\cdot c_{3,5}-\bar{c}_{1,4}\cdot a_{4},
$$
\[
\begin{aligned}
d(c_{1,3})&=&\bar{a}_1\cdot a_{2},\\
d(c_{1,4})&=&\bar{a}_1\cdot c_{2,4}+\bar{c}_{1,3}\cdot a_{3},\\
d(c_{2,4})&=&\bar{a}_2\cdot a_{3},\\
d(c_{2,5})&=&\bar{a}_2\cdot c_{3,5}+\bar{c}_{2,4}\cdot a_{4},\\
d(c_{3,5})&=&\bar{a}_3\cdot a_{4}. 
\end{aligned}
\]
\end{itemize}
\end{exam}

Therefore, from Definition~\ref{defMassey} and Example~\ref{examMassey} one can see easily that for a higher Massey product $\langle\alpha_{1},\ldots,\alpha_{k}\rangle$ to be defined it is necessary that all the Massey subproducts of consecutive elements in the $k$-tuple $(\alpha_{1},\ldots,\alpha_{k})$ are defined and \emph{vanish simultaneously}. If all the (consecutive) Massey subproducts vanish, but not simultaneously, then the whole Massey product may even not exist, see~\cite[Example I]{N}. 

\begin{rema}
In~\cite{M} matric Massey products were defined and studied, and it was also proved that differentials in the Eilenberg-Moore spectral sequence of a path loop fibration for any path connected simply connected space are completely determined by higher Massey products. However, the opposite statement is not true, see~\cite[Example II]{N}. 
\end{rema}

Let us consider a set of induced subcomplexes $K_{I_j}$ on pairwisely disjoint subsets of vertices $\{I_j\}$ for $1\leq j\leq k$ and their cohomology classes $\alpha_{j}\in\tilde{H}^{d(j)}(K_{I_j})\subset H^{m(j)}(\zk), 1\leq j\leq k$, where $m(j)=d(j)+|I_j|+1$ by Theorem~\ref{zkcoh}. If an $s$-fold Massey product ($s\leq k$) of consecutive classes $\langle\alpha_{i+1},\ldots,\alpha_{i+s}\rangle$ for $1\leq i+1<i+s\leq k$ is defined, then by Theorem~\ref{zkcoh} and Theorem~\ref{mgrad} $\langle\alpha_{i+1},\ldots,\alpha_{i+s}\rangle$ is a subset of
$$
\tilde{H}^{d(i+1,i+s)}(K_{I_{i+1}\sqcup\ldots\sqcup {I_{i+s}}})\subset H^{m(i+1,i+s)}(\zk),
$$
where $d(i+1,i+s)=d(i+1)+\ldots+d(i+s)+1$ and $m(i+1,i+s)=m(i+1)+\ldots+m(i+s)-s+2$. 

\begin{rema}
If all $m(j),1\leq j\leq k$ are odd numbers, then $m(i,j)=m(i)+\ldots+m(j)-(j-i+1)+2=\deg c_{i,j+1}+1$ for a defining system $C$ is an even number for all $1\leq i<i+1\leq j\leq k$. It follows, that in this case all the elements in $C$ have odd degrees in $R^{*}(K)$.
\end{rema}

Now our goal is to determine the conditions sufficient for a Massey product $\langle\alpha_{1},\ldots,\alpha_{k}\rangle$ of cohomology classes introduced above to be strictly defined. 

\begin{lemm}\label{strictMassey}
Suppose, in the notation above, $k\geq 3$ and
\begin{itemize}
\item[(1)] $\tilde{H}^{d(s,r+s)-1}(K_{I_{s}\sqcup\ldots\sqcup{I_{r+s}}})=0,1\leq s\leq k-r,1\leq r\leq k-2$;
\item[(2)] Any of the following two conditions holds: 
\begin{itemize}
\item[(a)] The $k$-fold Massey product $\langle\alpha_{1},\ldots,\alpha_{k}\rangle$ is defined, or
\item[(b)] $\tilde{H}^{d(s,r+s)}(K_{I_{s}\sqcup\ldots\sqcup{I_{r+s}}})=0,1\leq s\leq k-r,1\leq r\leq k-2$.
\end{itemize} 
\end{itemize}
Then the $k$-fold Massey product $\langle\alpha_{1},\ldots,\alpha_{k}\rangle$ is strictly defined.
\end{lemm}
\begin{proof}
Let us use induction on the order $k\geq 3$ of the Massey product. For the base case $k=3$ the condition (2b) implies that the 2-fold products $\langle\alpha_{1},\alpha_{2}\rangle$ and $\langle\alpha_{2},\alpha_{3}\rangle$ vanish simultaneously and, by Example~\ref{examMassey} (2), we get that the triple Massey product $\langle\alpha_{1},\alpha_{2},\alpha_{3}\rangle$ is defined. Moreover, by the condition (1) and Theorem~\ref{mgrad}, the indeterminacy in $\langle\alpha_{1},\alpha_{2},\alpha_{3}\rangle$ is trivial, thus this triple Massey product is strictly defined. 

Now assume the statement holds for Massey products of order less than $k\geq 4$. We first prove that conditions (1) and (2b) imply  $\langle\alpha_{1},\ldots,\alpha_{k}\rangle$ is defined. To make the induction step note that in Definition~\ref{defMassey} all the Massey products of consecutive elements of orders $2,\ldots,k-1$ for $k\geq 4$ are defined by the inductive assumption and are not only trivial (by (2b) and the multiplicative structure in $H^*(\zk)$ given by Theorem~\ref{zkcoh}) but also strictly defined, and thus contain only zero elements. Therefore, a defining system $C$ exists and a cocycle $a(C)$ from Definition~\ref{defMassey} (3) represents a cohomology class in $\langle\alpha_{1},\ldots,\alpha_{k}\rangle$.

Now we shall prove that this Massey product contains a unique element, that is, we need to prove that $[a(C)]=[a(C')]$ for any two defining systems $C$ and $C'$. By the inductive assumption we suppose that this statement is true for defined higher Massey products of order less than $k\geq 4$. We divide the induction step into 2 parts, in both we also proceed by induction.
\\
I. Let us determine a sequence of defining systems $C(1),\ldots,C(k-1)$ for $\langle\alpha_{1},\ldots,\alpha_{k}\rangle$ such that the following conditions hold:
\begin{itemize}
\item[(1)] $C(1)=C$;
\item[(2)] $c_{ij}(r)=c'_{ij}$, if $j-i\leq r$;
\item[(3)] $[a(C(r))]=[a(C(r+1))]$, for all $1\leq r\leq k-2$.
\end{itemize}
Note that (2) implies that $c_{i,i+1}(r)=a_{i}=a'_{i}$ for all $1\leq i\leq k$, and $C(k-1)=C'$. We use induction on $r\geq 1$ here to determine the defining systems $C(r)$. As $C(1)=C$ by (1), we need to prove the induction step assuming that $C(r)$ is already defined. 

Consider a cochain 
$$
b_{s}=c'_{s,r+s+1}-c_{s,r+s+1}(r)
$$
for $1\leq s\leq k-r$. By Definition~\ref{defMassey} one has: $d(b_{s})=d(c'_{s,r+s+1})-d(c_{s,r+s+1}(r))=\sum\limits_{p=s+1}^{r+s}\bar{c}'_{s,p}c'_{p,r+s+1}-
\sum\limits_{p=s+1}^{r+s}\bar{c}_{s,p}(r)c_{p,r+s+1}(r)=0$ by the property (2) of $C(r)$ above. Thus, $b_{s}$ is a cocycle and, therefore, by Theorem~\ref{zkcoh} and Theorem~\ref{mgrad}, $[b_s]\in\tilde{H}^{d(s,r+s)-1}(K_{I_{s}\sqcup\ldots\sqcup{I_{r+s}}})=0$ for any $1\leq s\leq k-r$, by the condition (1) of the statement we are to prove (as here $1\leq r\leq k-2$ by (3) above).  
\\
II. The construction of the defining system $C(r+1)$ will be completed if we determine a sequence of defining systems $C(r,s)$ ($0\leq s\leq k-r$) for $\langle\alpha_{1},\ldots,\alpha_{k}\rangle$, such that:
\begin{itemize}
\item[(1')] $C(r,0)=C(r)$;
\item[(2')] $c_{ij}(r,s)=c_{ij}(r)$, if $j-i\leq r$, and
$$
c_{ij}(r,s)=\begin{cases}
c_{ij}(r),&\text{if $i>s$; (*)}\\
c_{ij}(r)+b_{i},&\text{if $i\leq s$ (**)}
\end{cases}
$$
when $j-i=r+1\geq 2$;
\item[(3')] $[a(C(r,s))]=[a(C(r,s+1))]$, for all $0\leq s\leq k-r-1$. 
\end{itemize}
Note that condition (2')(**) for $j=i+r+1$ implies that $c_{ij}(r,k-r)=c_{ij}(r)+(c'_{i,r+1+i}-c_{i,r+1+i}(r))=c'_{ij}$ which is equal to $c_{ij}(r+1)$ by (2) above and thus we can define $C(r+1)=C(r,k-r)$ and the proof will be finished by induction.

Therefore, to complete the proof it suffices to determine a sequence of defining systems $C(r,s)$. We shall do it by induction on $s\geq 0$. The base case $s=0$ is done by (1') above. Now assume, we already constructed $C(r,s)$ and let us determine the defining system $C(r,s+1)$.

If $i>s+1$, then we can set $c_{ij}(r,s+1)=c_{ij}(r,s)$, see (2')(*). Similarly, we can also set $c_{ij}(r,s+1)=c_{ij}(r,s)$ if $j<s+r+2$, see (2')(**). Now suppose $1\leq i\leq s+1<s+2+r\leq j\leq k+1$ and let us determine a set of cochains $\{b_{ij}\}$ such that
$$
c_{ij}(r,s+1)=c_{ij}(r,s)+b_{ij}\eqno (***)
$$
by induction on $j-i\geq r+1$.

By (**) $c_{s+1,r+2+s}(r,s+1)=c_{s+1,r+2+s}(r)+b_{s+1}$ which is equal to $c_{s+1,r+2+s}(r,s)+b_{s+1}$ due to (*). Therefore, we can set $b_{s+1,r+2+s}=b_{s+1}$. Suppose by inductive assumption that the cochains $b_{ij}$ are already defined for $r+1\leq j-i<w$. Then (***) implies that 
$d(b_{ij})=d(c_{ij}(r,s+1))-d(c_{ij}(r,s))=\sum\limits_{p=i+1}^{j-1}(\bar{c}_{i,p}(r,s)+\bar{b}_{i,p})(c_{p,j}(r,s)+b_{p,j})-
\sum\limits_{p=i+1}^{j-1}\bar{c}_{i,p}(r,s)c_{p,j}(r,s)=\sum\limits_{p=i+1}^{s+1}\bar{c}_{i,p}(r,s)b_{p,j}+
\sum\limits_{p=r+s+2}^{j-1}\bar{b}_{i,p}c_{p,j}(r,s)$, where the last equality holds because $\sum\limits_{p=i+1}^{j-1}\bar{b}_{i,p}b_{p,j}=0$, since $b_{p,j}=0$ when $p>s+1$, and $b_{i,p}=0$ when $p<r+2+s$, and one gets the following equality for $r+1\leq j-i<w$:
$$
d(b_{ij})=\sum\limits_{p=i+1}^{s+1}\bar{c}_{i,p}(r,s)b_{p,j}+
\sum\limits_{p=r+s+2}^{j-1}\bar{b}_{i,p}c_{p,j}(r,s)\eqno (1.1)
$$   

Then for $j-i=w$ the right hand side of (1.1) is a cocycle $a$ representing an element $\alpha=[a]$ in 
$$
-\langle\alpha_{i},\ldots,\alpha_{s},[b_{s+1}],\alpha_{s+r+2},
\ldots\alpha_{j-1}\rangle \eqno (1.2)
$$ 
as can be shown by induction on $j-i\geq r+2$. Indeed, for $j-i=r+2$ and $1\leq i\leq s+1<s+2+r\leq j\leq k+1$ one has two cases: (1) $i=s+1,j=s+r+3$ and the right hand side of (1.1) takes the form $\bar{b}_{s+1,r+s+2}c_{s+r+2,s+r+3}=\bar{b}_{s+1}a_{s+r+2}$ which represents $-\langle[b_{s+1}],\alpha_{s+r+2}\rangle$; (2) $i=s,j=s+r+2$ and the right hand side of (1.1) takes the form $\bar{c}_{s,s+1}b_{s+1,s+r+2}=\bar{a}_{s}b_{s+1}$ which represents $-\langle\alpha_{s},[b_{s+1}]\rangle$. The induction step follows from the formula (1.1) itself, Definition~\ref{defMassey}, and the inductive assumption.   

As $[b_{s+1}]=0$, we conclude that the Massey product in (1.2) is trivial. Moreover, as $r\geq 1$ one can easily see that its order is $<(j-1)-i+1=j-i\leq k$ and we can apply the inductive assumption on $k$ to this Massey product, since 
$[b_{s+1}]\in\tilde{H}^{\beta}(K_{I_{s+1}\sqcup\ldots\sqcup{I_{r+s+1}}})$ for $\beta=d(s+1,r+s+1)-1=d(s+1)+\ldots+d(r+s+1)$. Therefore, by the assumption of induction on $k$ one has that the Massey product
$$
0\in\langle\alpha_{i},\ldots,\alpha_{s},[b_{s+1}],\alpha_{s+r+2},
\ldots,\alpha_{j-1}\rangle
$$ 
is strictly defined, that is, it contains only zero. It follows that (1.1) has a solution for $j-i=w$, and thus for all $1\leq i<j\leq k+1$ the equality (1.1) gives:
$$
d(b_{ij})=\sum\limits_{p=i+1}^{j-1}\bar{c}_{i,p}(r,s+1)c_{p,j}(r,s+1)-
\sum\limits_{p=i+1}^{j-1}\bar{c}_{i,p}(r,s)c_{p,j}(r,s).
$$
The above formula means that $C(r,s+1)$ is also a defining system for $\langle\alpha_{1},\ldots,\alpha_{k}\rangle$ and that $[a(C(r,s+1))]=[a(C(r,s))]$ (when $j-i=k$ in the above formula).
This finishes the proof by induction on the order $k$ of the Massey product.
\end{proof}

\begin{rema}
Note that if $I_{j}\in\MF(K)$ for $1\leq j\leq k$ are pairwisely disjoint sets of vertices in $[m]$, then $K_{I_j}=\partial\Delta^{|I_j|-1}\simeq S^{|I_{j}|-2}$ and for its top cohomology class $\alpha_j$ one has that its degree $m(j)=d(j)+|I_j|+1=2|I_j|-1$ is always odd. In order to obtain a defined and/or strictly defined $k$-fold Massey product $\langle\alpha_{1},\ldots,\alpha_{k}\rangle$ one needs to check the condition (2b) (resp., both (1) and (2b)) of Lemma~\ref{strictMassey}, that is, to calculate $2+\ldots+(k-1)=\frac{k(k-1)}{2}-1$ (resp., $k(k-1)-2$) multigraded Betti numbers of $K$, see Theorem~\ref{mgrad}.
\end{rema}

\begin{exam}\label{examLemma}
(1) Suppose $K$ is a nerve complex of a hexagon $P=P_6$. Thus, $K$ is a 1-dimensional simplicial complex on the vertex set $[6]$ and 
$$
\MF(K)=\{(1,3),(1,4),(1,5),(2,4),(2,5),(2,6),(3,5),(3,6),(4,6)\}.
$$
Consider the three minimal non-faces $I_{j}=(j,j+3)\in\MF(K)$ for $1\leq j\leq 3$ and the corresponding cohomology classes $\alpha_{j}=[v_{j}u_{j+3}]\in{H}^{3}(\zp)$. Obviously, there exist two defining systems for $\langle\alpha_{1},\alpha_{2},\alpha_{3}\rangle\subset{H}^{8}(\zp)$, which give a zero element and a generator of $H^{8}(\zp)$ as elements of the triple Massey product above, respectively. Therefore, $\langle\alpha_{1},\alpha_{2},\alpha_{3}\rangle$ is defined, but not strictly defined. One can see easily that in this case
$$
\rk\tilde{H}^{0}(K_{I_{1}\sqcup I_{2}})=\rk\tilde{H}^{0}(K_{I_{2}\sqcup I_{3}})=1,
$$ 
therefore, the condition (1) of Lemma~\ref{strictMassey} is not satisfied.\\
(2) Suppose $K$ is one of the five obstruction graphs introduced in~\cite[Theorem 6.1.1]{DS}. One can easily check that each of them satisfies both the conditions (1) and (2b) of Lemma~\ref{strictMassey}, thus a nontrivial triple Massey product of 3-dimensional classes in $H^*(\zk)$ is strictly defined (cf.~\cite[\S6.2]{DS}).
\end{exam}

Now we shall apply Lemma~\ref{strictMassey} to the problem of finding a family of generalized moment-angle manifolds $\zp^{J}$ that contains an $l$-connected manifold with an $r$-fold strictly defined nontrivial Massey product in $H^*(\zp^{J})$ for any prescribed $l,r\geq 2$. 

\begin{defi}\label{deftrunc}
Suppose $n\geq 2$. Let us define simplicial complexes $K(n)$ and $\bar{K}(n)$ on the vertex set $[2n]=\{1,2,\ldots,2n\}$ with the following sets of minimal non-faces
$$
\MF(K(n))=\{(k,n+k+i), 0\leq i\leq n-2,1\leq k\leq n-i\}
$$
and
$$
\MF(\bar{K}(n))=\{(k,n+k+i), 0\leq i\leq n-1,1\leq k\leq n-i\},
$$
respectively. We denote by
$$
K(n,s)=K(n)(s,\ldots,s,1,\ldots,1)
$$ 
and 
$$
\bar{K}(n,s)=\bar{K}(n)(s,\ldots,s,1,\ldots,1)
$$ 
for $s\geq 1$ their simplicial multiwedges in which the first $n$ vertices are wedged and the last $n$ vertices remain unchanged.

Finally, we denote by $P(n)$ an $n$-dimensional 2-truncated cube with $m=m(P(n))=\frac{n(n+3)}{2}-1$ facets obtained from the cube $I^n$ with the facets $F_{1},\ldots,F_{2n}$, where $F_{i}\cap F_{n+i}=\varnothing$ for $1\leq i\leq n$ such that the induced subcomplex in $K_{P(n)}$ on the vertex set $[2n]$ is combinatorially equivalent to $K(n)$, see~\cite[Definition, p.15]{L1}. Similarly, we denote by $P(n,s)=P(n)(J_{ns})$ a simple polytope with $d(J_{ns})$ facets, where 
$$
J_{ns}=(\underbrace{s,\ldots,s}\limits_{n},\underbrace{1,\ldots,1}\limits_{n},j_{2n+1,n,s},\ldots,j_{m,n,s}),
$$
the induced subcomplex in $K_{P(n,s)}$ on the vertex set
$$
\{11,\ldots,1s,\ldots,n1,\ldots,ns,n+1,\ldots,2n\}
$$
is combinatorially equivalent to $K(n,s)$, and $\{j_{p,n,s}\}$ for $p>2n$ are arbitrary positive integers satisfying the property: $j_{p,n,1}=1$ and for any $I\in\MF(K_{P(n,s)})$ one has $|I|\geq s+1$, providing $K_{P(n,s)}$ is $s$-connected for all $s\geq 1$.  
\end{defi}

\begin{theo}\label{mainMassey}
Consider a family of smooth closed manifolds with a compact torus action: 
$$
\mathcal{F}=\{\zp\,|\,P=P(n,s), n\geq 2, s\geq 1\}.
$$
Then for any given $l,r\geq 2$ there is an $l$-connected manifold $M\in\mathcal{F}$ with a strictly defined nontrivial $r$-fold Massey product in $H^*(M)$.
\end{theo}
\begin{proof}
First note, that the nerve complex of a simple polytope $P=P(r,[\frac{l+1}{2}])$
for $r,l\geq 2$ is $[\frac{l+1}{2}]$-connected and, thus, $\zp$ is $l$-connected. Therefore, it suffices to prove that $\zp$ has a strictly defined nontrivial Massey product of order $n\geq 2$ for a simple polytope $P=P(n,s)$ with $s\geq 1$.

Consider the following set of $n$ minimal non-faces $I_{j}\in\MF(K)$ of its nerve complex $K=K_{P(n,s)}$: $\{I_{j}=(j1,\ldots,js,n+j)|\,1\leq j\leq n\}$ and the top cohomology classes $\alpha_{j}$ of the induced subcomplexes $K_{I_j}$: $\alpha_{j}\in\tilde{H}^{s-1}(K_{I_j})$. Then in the notation of Lemma~\ref{strictMassey} one has: $d(j)=s-1$ and $m(j)=2s+1$. To apply Lemma~\ref{strictMassey} to the ordered set of cohomology classes $(\alpha_{1},\ldots,\alpha_{n})$ we need to check the conditions (1) and (2b) of the lemma for them.

By Definition~\ref{deftrunc} the induced subcomplex of $K$ on the vertex set
$$
\{11,\ldots,1s,\ldots,n1,\ldots,ns,n+1,\ldots,2n\}
$$
coincides with $K(n,s)$. Observe that the induced subcomplex of $K$ on a subset of its vertex set of the following form
$$
\{i1,\ldots,is,(i+1)1,\ldots,(i+1)s,\ldots,j1,\ldots,js,n+i,n+i+1,\ldots,n+j\}
$$
for $1\leq i\leq j\leq n, j-i\leq n-2$ is combinatorially equivalent to $\bar{K}(j-i+1,s)$. 

Therefore, to apply Lemma~\ref{strictMassey} we need to show that for $\bar{K}=\bar{K}(n,s)$ with the same minimal non-faces $I_{j}, 1\leq j\leq n$ as above, the equalities 
$$
\tilde{H}^{d(v,w)}(\bar{K}_{I_{v}\sqcup\ldots\sqcup{I_{w}}})=0
$$
and
$$
\tilde{H}^{d(v,w)-1}(\bar{K}_{I_{v}\sqcup\ldots\sqcup{I_{w}}})=0
$$
hold for all $1\leq v<w\leq n$ and $n\geq 2$.

Again, it is obvious from the definition of $\bar{K}(n,s)$ that such an induced subcomplex $\bar{K}_{I_{v}\sqcup\ldots\sqcup{I_{w}}}$ is combinatorially equivalent to $\bar{K}(w-v+1,s)$ and $d(v,w)=d(v)+\ldots+d(w)+1=(w-v+1)(s-1)+1$. Thus, finally, we need to prove that  
$$
\tilde{H}^{n(s-1)+1}(\bar{K}(n,s))=0
$$
and 
$$
\tilde{H}^{n(s-1)}(\bar{K}(n,s))=0,
$$
where we set $n=w-v+1$.

The simplicial complex $\bar{K}(n,s)$ has $(s+1)n$ vertices and by Theorem~\ref{hoch} the rank of the cohomology group $\tilde{H}^{n(s-1)+1}(\bar{K}(n,s))$ is equal to the algebraic Betti number  
$$
\beta^{-i,2j}(\bar{K}(n,s))=\beta^{-(2n-2),2n(s+1)}(\bar{K}(n,s)),
$$
where $j=(s+1)n$ and $j-i-1=n(s-1)+1$. Similarly, the rank of the cohomology group $\tilde{H}^{n(s-1)}(\bar{K}(n,s))$ is equal to the algebraic Betti number  
$$
\beta^{-i,2j}(\bar{K}(n,s))=\beta^{-(2n-1),2n(s+1)}(\bar{K}(n,s)).
$$

Then due to a multigraded version of~\cite[Corollary 7.9]{Ay} and Theorem~\ref{hoch} we get that 
$$
\beta^{-(2n-2),2n(s+1)}(\bar{K}(n,s))=\beta^{-(2n-2),2(s,\ldots,s,1,\ldots,1)}(\bar{K}(n,s))
$$ 
is equal to 
$$
\beta^{-(2n-2),2\cdot{2n}}(\bar{K}(n))=\rk\tilde{H}^{1}(\bar{K}(n))=0,
$$
since $\bar{K}(n)$ is a contractible chain of simplices of dimension $n-1$, glued along their facets, one by one. Similarly, one has that 
$$
\beta^{-(2n-1),2n(s+1)}(\bar{K}(n,s))=\beta^{-(2n-1),2(s,\ldots,s,1,\ldots,1)}(\bar{K}(n,s))
$$ 
is equal to 
$$
\beta^{-(2n-1),2\cdot{2n}}(\bar{K}(n))=\rk\tilde{H}^{0}(\bar{K}(n))=0,
$$ 
since $\bar{K}(n)$ is connected. This observation finishes the proof of that the Massey $n$-product $\langle\alpha_{1},\ldots,\alpha_{n}\rangle$ is defined and strictly defined in $H^*(\zp)$ for $P=P(n,s)$ and $s\geq 1$. 

Consider a representative $u_{j1}v_{j2}\ldots v_{js}v_{n+j}\in R^*(K_P)$ for $\alpha_{j}\in H^*(\zp)$.
To finish the proof of our theorem we must find the unique nonzero element in $\langle\alpha_{1},\ldots,\alpha_{n}\rangle$. A straightforward calculation shows that these representatives produce a defining system $C$ with 
$$
a(C)=v_{11}\ldots v_{1s}v_{2n}v_{22}\ldots v_{2s}\ldots v_{n2}\ldots v_{ns}u_{21}\ldots u_{n1}u_{n+1}\ldots u_{2n-1}
$$ 
(up to sign) being a product of two cocycles:
$$
v_{2n}v_{22}\ldots v_{2s}\ldots v_{n2}\ldots v_{ns}u_{21}\ldots u_{n1}
$$
and
$$
v_{11}\ldots v_{1s}u_{n+1}\ldots u_{2n-1},
$$ 
(up to signs), which gives a nonzero (and decomposable) element $[a(C)]\in H^*(\zp)$. 
\end{proof}

\begin{rema}
Note that by Definition~\ref{deftrunc} $K_{P(n,s)}=K_{P(n)}(J_{ns})$. Thus, if
$(\mathbb{D},\mathbb{S})=(\b{D}^{2j_{i,n,s}},\b{S}^{2j_{i,n,s}-1})$ ($1\leq i\leq m(n,s)$), then, due to Theorem~\ref{homeoBBCG}, one has: $\mathcal Z_{P(n,s)}\cong (\mathbb{D},\mathbb{S})^{K_{P(n)}}$. Therefore, Theorem~\ref{mainMassey} implies that there exists a family of generalized moment-angle manifolds over 2-truncated cubes $P(n)$ having strictly defined nontrivial $n$-fold Massey products in cohomology for any $n\geq 2$.
\end{rema}

\section{Applications}

In this section we are going to introduce applications of our main results to the theory of Stanley--Reisner rings with nontrivial Massey products in their Koszul homology and to rational homotopy theory of polyhedral products.

We start with a discussion of higher nontrivial Massey products in Tor-algebras of the type $\Tor^{S}_{*}(R,\ko)$, where $S$ is a polynomial ring and $R=S/I$ is a monomial ring. This area of homological algebra dates back to the pioneering work of Golod~\cite{Go}, in which it was proved that a local (Noetherian commutative) ring $R$ is \emph{Golod}, that is, the product and all the higher Massey products in $\Tor^{S}_{*}(R,\ko)$ are trivial ($\ko=R/m$), if and only if the Poincar\'e series of $R$ is given by a rational function of a certain type. In the case of a face ring $\ko[K]$ and its Tor-algebra $\Tor_{\ko[v_1,\ldots,v_m]}^{*,*}(\ko[K],\ko)$ the analogous result was proved by Grbi\'c and Theriault~\cite[Theorem 11.1]{GT}. 

As a counterexample to a claim by Berglund and J\"ollenbeck, Katth\"an~\cite{Kat} constructed a simplicial complex $K$ such that all products (of elements of positive degrees) are trivial in $\Tor_{\ko[v_1,\ldots,v_m]}^{*,*}(\ko[K],\ko)$, but $\ko[K]$ is not Golod having a nontrivial triple Massey product in its Koszul homology $\Tor_{\ko[v_1,\ldots,v_m]}^{*,*}(\ko[K],\ko)$. 

Furthermore, it was also shown in~\cite[Theorem 4.1]{Kat} that a monomial ring $R$ is Golod if all $r$-fold Massey products vanish for all $r=\max(2,\reg(R)-2)$. On the other hand, Frankhuizen~\cite{Frank} has recently proved that for a monomial ring $R$ whose minimal free resolution is rooted, $R$ is Golod if and only if the product on $\Tor^{S}_{*}(R,\ko)$ vanishes.

We show that the next result holds for the higher nontrivial Massey products in $\Tor_{\ko[v_1,\ldots,v_m]}^{*,*}(\ko[K],\ko)$.

\begin{theo}\label{coromainMassey}
Cohomology algebra $H^*(\zp;\ko)\cong\Tor_{\ko[v_1,\ldots,v_m]}^{*,*}(\ko[K],\ko)$ for $P=P(n,s), K=K_{P(n,s)}, m=d(J_{ns}), s\geq 1$ contains a strictly defined nontrivial $r$-fold Massey product for any $2\leq r\leq n$.\\
Moreover, the same holds for $\Tor_{\ko[v_1,\ldots,v_m]}^{*,*}(\ko[K],\ko)$, where $K=K(n)(J)$ is any simplicial multiwedge over $K(n)$, $m=d(J)$.
\end{theo}
\begin{proof}
Observe that for the simplicial complex $K=K(n,s)$ the following combinatorial equivalence holds
$$
K_{I_{1}\sqcup\ldots\sqcup I_{r-1}\sqcup I_{n}}\simeq K(r,s)
$$
for all $2\leq r\leq n$, which shows that, in the notation of Theorem~\ref{mainMassey}, the $r$-fold Massey product $\langle\alpha_{1},\ldots,\alpha_{r-1},\alpha_{n}\rangle$ is strictly defined and nontrivial in $H^*(\zk;\ko)$ by Theorem~\ref{mainMassey}.
The rest follows from the above observation applying the same argument as in the proof of Theorem~\ref{mainMassey}.
\end{proof}

\begin{rema}
For $s=1,r=n$ Theorem~\ref{coromainMassey} implies that the $n$-fold Massey product of 3-dimensional cohomology classes in~\cite[Theorem 4.5]{L1} is not only defined and nontrivial but also strictly defined, containing the unique element $[v_{1}v_{2n}u_{2}\ldots u_{2n-1}]$ (up to sign). 
\end{rema}

\begin{prop}\label{MasseyDegrees}
Given a set $\{k_{i}\geq 3|\,1\leq i\leq n\}$ of odd integers, there exists a polytopal triangulated sphere $K=K_P$ such that a strictly defined nontrivial Massey product $\langle\alpha_{1},\ldots,\alpha_{n}\rangle$ with $\dim\alpha_{i}=k_{i},1\leq i\leq n$ is defined in $H^*(\zp;\ko)\cong\Tor_{\ko[v_1,\ldots,v_m]}^{*,*}(\ko[K],\ko)$.
\end{prop}
\begin{proof}
Let $k_{i}=2d_{i}+1$, where $d_{i}\geq 1$, and consider an $m$-tuple of positive integers $J=(d_{1},\ldots,d_{n},1,\ldots,1)$. Let $P=P(n)(J)$, $K=K_P$ and $\alpha_{i}$ be a generator of the cohomology group $\tilde{H}^{d_{i}-1}(K_{\{i1,\ldots,id_{i},n+i\}};\ko)\cong\tilde{H}^{d_{i}-1}(S^{d_{i}-1};\ko)$. Analogously to the proof of Theorem~\ref{mainMassey}, the conditions (1) and (2b) of Lemma~\ref{strictMassey} are valid for $(\alpha_{1},\ldots,\alpha_{n})$, since 
by the multigraded version of~\cite[Corollary 7.9]{Ay} one has ($j=\sum\limits_{i=1}^{n}d_{i}+n$ and $j-i-1=d(1,n)=\sum\limits_{i=1}^{n}d_{i}-n+1$ for condition (2b), or $j-i-1=d(1,n)-1=\sum\limits_{i=1}^{n}d_{i}-n$ for condition (1); we drop $\ko$ from the notation in what follows):
$$
\beta^{-i,2j}(\bar{K}(n)(J))=\beta^{-(2n-2),2(d_{1},\ldots,d_{n},1\ldots,1)}(\bar{K}(n)(J))=\beta^{-(2n-2),2\cdot 2n}(\bar{K}(n)),
$$
the latter equals to $\rk\tilde{H}^1(\bar{K}(n))=0$ (for condition (2b)), and 
$$
\beta^{-i,2j}(\bar{K}(n)(J))=\beta^{-(2n-1),2(d_{1},\ldots,d_{n},1\ldots,1)}(\bar{K}(n)(J))=\beta^{-(2n-1),2\cdot 2n}(\bar{K}(n)),
$$
the latter equals to $\rk\tilde{H}^0(\bar{K}(n))=0$ (for condition (1)). The rest of the proof goes similarly to that of Theorem~\ref{mainMassey}.
\end{proof}

Next we are going to introduce the applications of our results on nontrivial higher Massey products to studying (rational) formality of (generalized) moment-angle manifolds. For more details on rational homotopy theory we refer the reader to the monographs~\cite{f-h-t01, f-o-t08}. The applications of rational homotopy theory in toric topology can be found in~\cite{BBCG14, DS, F-T}. 

We now turn to the definition of formal differential graded algebras and (rationally) formal topological spaces.

\begin{defi}\label{defFormal}
(1) A differential graded algebra $[A,d]$ is called \emph{formal} if there exists a chain of quasi-isomorphisms (i.e., a weak equivalence in CDGA category) between $[A,d]$ and its cohomology algebra with a trivial differential $[H^{*}(A),0]$.
(2) A space $X$ is called \emph{(rationally) formal} if its Sullivan-de Rham algebra of PL-differential forms $A_{PL}(X;\mathbb{Q})$ (the Sullivan minimal model) with the de Rham differential $d$ is formal in the sense of (1), that is, if it is weakly equivalent in the CDGA category to the singular cohomology algebra $H^{*}(X;\mathbb{Q})$ of $X$ with a trivial differential.
\end{defi}

\begin{rema}
Sullivan~\cite{sull77} proved that formality of $X$ implies that all (ordinary) higher Massey products in its cohomology algebra $H^*(X)$ are vanishing. However, the opposite statement does not hold in general.
\end{rema}

\begin{exam}
Among the examples of formal spaces are:
\begin{itemize}
\item[(1)] spheres; H-spaces and, in particular, Eilenberg-MacLane spaces $\K(\pi,n)$ for $n>1$; symmetric spaces;
\item[(2)] compact connected Lie groups $G$ and their classifying spaces $BG$ ~\cite{sull77};
\item[(3)] compact K\"ahler manifolds~\cite{D-G-M-S} and, in particular, projective toric varieties;
\item[(4)] quasitoric manifolds~\cite{P-R};
\item[(5)] polyhedral products of the type $X^{K}:=(X,*)^K$, see~\cite[Theorem 8.1.2]{TT}. 
\end{itemize}
The last case includes the following particular examples:
\begin{itemize}
\item[(a)] Davis-Januszkiewicz spaces $\Dj(K)=(\mathbb{C}P^{\infty})^K$ are formal, see~\cite{no-ra05};
\item[(b)] Suppose $K=\Delta(\Gamma)$ is a clique complex for a simple graph $\Gamma$. Then the Eilenberg-MacLane space $\K(RA_{\Gamma},1)=(S^1)^{K}$ for the right-angled Artin group $RA_{\Gamma}$ is formal, see~\cite{P-S};
\item[(c)] Similarly, the Eilenberg-MacLane space $\K(RC_{\Gamma},1)=(\mathbb{R}P^{\infty})^{K}$, for the right-angled Coxeter group $RC_{\Gamma}$ is formal. 
\end{itemize}
Moreover, formality is preserved under products and wedges of spaces. A connected sum of formal manifolds is also formal.
\end{exam}

Baskakov~\cite{BaskM} constructed the first example of a polyhedral product, which is not formal. Namely, he introduced a family of moment-angle-complexes $\zk$ over triangulated spheres $K$ having a nontrivial triple Massey product in $H^*(\zk)$. Note that by Theorem~\ref{zkcoh} for any triangulated sphere $K$ its moment-angle-complex $\zk$ is 2-connected and when $K$ is flag, the third Betti number is equal to the number of minimal non-faces in $K$: $b^{3}(\zk)=\beta^{-1,4}(K)=|\MF(K)|$. 

In what follows we shall give a partial solution of the following open problem, see~\cite[Problem 8.28]{BP04}:

\begin{prob}\label{probFormal}
Determine a class of simplicial complexes $K$ (simple polytopes $P$), for which a differential graded algebra $R^*(K)$ (a moment-angle-complex $\zk$, a moment-angle manifold $\zp$) is (rationally) formal. 
\end{prob}

We are going to state the answer to the above problem when $P$ is a graph-associahedron. To do this, we first recall a definition of this class of flag simple polytopes.

\begin{defi}\label{nest}
A {\emph{building set}} on $[n+1]=\{1,2,\ldots,n+1\}$ for $n\geq 2$ is a family of nonempty subsets $B=\{S\subseteq [n+1]\}$, such that: 1) $\{i\}\in B$ for all $1\leq i\leq n+1$, 2) if $S_1\cap S_2\ne\varnothing$, then $S_1\cup S_2\in B$. A building set is called {\emph{connected}} if $[n+1]\in B$.

For any building set $B$ on $[n+1]$ we get an $n$-dimensional simple polytope $P_B$ called a \emph{nestohedron (on a building set $B$)} as a Minkowski sum of simplices:
$$
P_{B}=\sum\limits_{S\in B}\Delta_{S},\quad\Delta_{S}=\conv\{{e}_j|\,j\in S\}\subset \mathbb R^{n+1}.
$$ 
Note that facets of $P_{B}$ are in 1-1 correspondence with non-maximal elements $S$ in $B$ (\cite{FS},\cite[Proposition 1.5.11]{TT}).
\end{defi}

\begin{exam}\label{Bcube,simplex}
For a combinatorial $n$-simplex $P$ the subset of $2^{[n+1]}$ consisting of all the singletons $\{i\},1\leq i\leq n+1$ and the whole set $[n+1]$ is a connected building set $B$ such that $P=P_{B}$ when $n\geq 2$. \\
For a combinatorial $n$-cube $P$ the subset $B$ of $2^{[n+1]}$ consisting of
$$
\{1\},\ldots,\{n+1\},\{1,2\},\{1,2,3\},\ldots,[n+1]
$$
is a connected building set such that $P=P_B$ when $n\geq 2$.
\end{exam}

Due to the result of Buchstaber and Volodin~\cite[Proposition 6.1, Theorem 6.5]{bu-vo11}, a nestohedron is flag if and only if it is a \emph{2-truncated cube}, i.e. it can be obtained as a result of a sequence of codimension 2 face cuttings by hyperplanes of general position, starting with a cube. The next family of polytopes introduced by Carr and Devadoss~\cite{CD} consists of flag nestohedra and, therefore, can be represented as 2-truncated cubes.  

\begin{defi}\label{defiGA}
Suppose $\Gamma$ is a simple graph on the vertex set $[n+1]$. \emph{A graphical building set} $B(\Gamma)$ on $[n+1]$ consists of such $S$ that the induced subgraph $\Gamma_{S}$ on the vertex set $S\subset [n+1]$ is a connected graph.\\
The resulting nestohedron $P_{\Gamma}=P_{B(\Gamma)}$ is called a \emph{graph-associahedron}.
\end{defi}

Among graph-associahedra, there are such polytopes arising in different areas of Mathematics as permutohedra $\Pe^n$ ($\Gamma$ is a complete graph), associahedra $\As^n$ (or, Stasheff polytopes~\cite{S}; $\Gamma$ is a path graph), stellahedra $\St^n$ ($\Gamma$ is a stellar graph), and cyclohedra $\Cy^n$ (or, Bott-Taubes polytopes~\cite{BT}; $\Gamma$ is a cycle graph). 

Now we are going to give a solution of the Problem~\ref{probFormal} in the class of graph-associahedra and all the polytopes that can be obtained from graph-associahedra using $J$-construction. Namely, we shall prove that certain conditions on the combinatorics of $\Gamma$ are necessary and sufficient for the moment-angle manifold $\zp$ of such a polytope $P$ to be (rationally) formal. In order to prove this result, we need the next statement.

\begin{lemm}\label{MasseyGA}
Suppose $P=P_{\Gamma}$ is a graph-associahedron.
Then the following two statements hold.
\begin{itemize}
\item[(1)] There exists a nontrivial strictly defined triple Massey product $\langle\alpha_{1},\alpha_{2},\alpha_{3}\rangle$ of 3-dimensional cohomology classes in $H^{*}(\zp)$ if and only if there is a connected component of $\Gamma$ on $m\geq 4$ vertices, which is different from a complete graph $K_{4}$;
\item[(2)] If $P=\Pe^3$ then there exists a nontrivial strictly defined triple Massey product $\langle\alpha_{1},\alpha_{2},\alpha_{3}\rangle$ in $H^*(\zp)$ with $\dim\alpha_{1}=\dim\alpha_{3}=5$, $\dim\alpha_{2}=3$.
\end{itemize} 
\end{lemm}
\begin{proof}
The first statement was proved in~\cite[Proposition 4.2]{L1}; it follows also from the proof of~\cite[Theorem 6.1.1]{DS} that a nontrivial triple Massey product of 3-dimensional classes in $H^*(\zk)$ is always strictly defined, cf. Example~\ref{examLemma}\,(2). 

To prove the second statement, consider the following induced subcomplexes in the nerve complex $K=K_P$:
$J_{1}$ on the vertex set $\{2,3,9,12\}$, $J_{2}$ on the vertex set $\{1,14\}$, and $J_{3}$ on the vertex set $\{5,7,8,13\}$, and the following cohomology classes: $\alpha=[v_{12}u_{2}u_{3}u_{9}]\in\tilde{H}^{0}(K_{J_1})\subset H^{5}(\zp)$, $\beta=[v_{1}u_{14}]\in\tilde{H}^{0}(K_{J_2})\subset H^{3}(\zp)$, and $\gamma=[v_{5}u_{7}u_{8}u_{13}]\in\tilde{H}^{0}(K_{J_3})\subset H^{5}(\zp)$. 
One can see easily that the induced subcomplexes $K_{J_{1}\sqcup J_{2}}$ and $K_{J_{2}\sqcup J_{3}}$ in $K$ are contractible, see Figure~\ref{permut}, in which the facets of $P$ that can be seen are marked with red numbers and those, that remain unseen, are marked with light blue ones. 

\begin{figure}[h]
\includegraphics[scale=0.8]{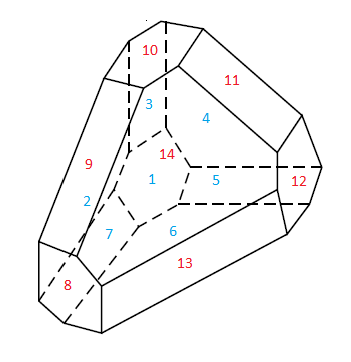}
\caption{3-dimensional permutohedron $P$.}
\label{permut}
\end{figure}

Thus, by Theorem~\ref{zkcoh}, $\alpha\cdot\beta\in\tilde{H}^{1}(K_{J_{1}\sqcup J_{2}})$ and $\beta\cdot\gamma\in\tilde{H}^{1}(K_{J_{2}\sqcup J_{3}})$ are trivial. Therefore, the triple Massey product $\langle\alpha,\beta,\gamma\rangle$ is defined and a straightforward calculation shows that (up to sign) a nonzero 12-dimensional cohomology class 
$$
[v_{5}v_{12}u_{1}u_{2}u_{3}u_{7}u_{8}u_{9}u_{13}u_{14}]
$$
belongs to $\in\langle\alpha,\beta,\gamma\rangle$.  
The indeterminacy is lying in the linear subspace 
$$
\alpha\cdot\tilde{H}^{0}(K_{J_{2}\sqcup J_{3}})+\gamma\cdot\tilde{H}^{0}(K_{J_{1}\sqcup J_{2}})=0,
$$
since both the induced subcomplexes $K_{J_{1}\sqcup J_{2}}$ and $K_{J_{2}\sqcup J_{3}}$ are connected. We conclude that the triple Massey product $\langle\alpha,\beta,\gamma\rangle$ is strictly defined and nontrivial.
\end{proof}

\begin{theo}\label{mainFormal}
Suppose $P=P_{\Gamma}$ is a graph-associahedron of dimension $n\geq 2$ with $m$ facets. Then the following statements are equivalent:
\begin{itemize}
\item[(a)] Each of the connected components of $\Gamma$ is either a vertex, or a segment, or a path on 3 vertices, or a cycle on 3 vertices;
\item[(b)] $P$ is a product of segments $I^{1}$, pentagons $P_5$, and hexagons $P_6$;
\item[(c)] There are no defined nontrivial triple Massey products in $H^*(\zp)\cong \Tor_{\mathbb{Z}[v_{1},\ldots,v_{m}]}(\mathbb{Z}[P],\mathbb{Z})$; 
\item[(d)] $\zp$ is a (rationally) formal space;
\item[(e)] $\mathcal Z_{P(J)}$ is a (rationally) formal space for any $m$-tuple of positive integers $J$.
\end{itemize}
A formal $\zp$ over a graph-associahedron $P$ is diffeomorphic to a product of a number of spheres $M_{1}=S^3$ and connected sums of products of spheres:
$$
M_{2}=(S^{3}\times S^{4})^{\#5}, M_{3}=(S^{3}\times S^{5})^{\#9}\#(S^{4}\times S^{4})^{\#8}.
$$ 
\end{theo}
\begin{proof}
The equivalence $(a)\Leftrightarrow(b)$ follows directly from the Definition~\ref{defiGA}. 

By Lemma~\ref{MasseyGA}, if $\Gamma$ has a connected component on $m\geq 4$ vertices, then the corresponding moment-angle manifold $\mathcal Z_{P_\Gamma}$ has a nontrivial strictly defined triple Massey product in cohomology, thus $\mathcal Z_{P_\Gamma}$ can not be formal. Otherwise, by (b) $P$ is combinatorially equivalent to a product of segments $I^1$, polygons $P_5$, and hexagons $P_6$. By the computation of diffeomorphism types of $\zp$ for polygons $P$, see~\cite{bo-me06,mcga79}, one has: $\mathcal Z_{I^1}=S^3$, $\mathcal Z_{P_5}=(S^{3}\times S^{4})^{\#5}$, $\mathcal Z_{P_6}=(S^{3}\times S^{5})^{\#9}\#(S^{4}\times S^{4})^{\#8}$, and $\zp$ is a product of the latter manifolds, since $\mathcal Z_{P_{1}\times P_{2}}$ is diffeomorphic to $\mathcal Z_{P_1}\times\mathcal Z_{P_2}$ for any simple polytopes $P_{1}$ and $P_{2}$, and $P_{\Gamma_{1}\sqcup\Gamma_{2}}=P_{\Gamma_1}\times P_{\Gamma_2}$ by Definition~\ref{nest} and Definition~\ref{defiGA}. Such a manifold $\zp$ has no nontrivial higher Massey products in cohomology, which proves the equivalence $(a)\Leftrightarrow (c)$, and, moreover, is (rationally) formal, which proves the equivalence $(a)\Leftrightarrow (d)$, as well as the final part of the theorem about diffeomorphism types of formal moment-angle manifolds $\zp$ over graph-associahedra $P=P_\Gamma$.

Obviously, (e) implies (d) (for $J=(1,\ldots,1)$ one has $P(J)=P$).
Suppose (d) holds. It follows from the Construction~\ref{simpmultwedge} that for any nonempty simplicial complexes $K_{1}$ on $[m_1]$ and $K_{2}$ on $[m_2]$ with no ghost vertices the simplicial multiwedge operation preserves joins: $(K_{1}*K_{2})(J)=K_{1}(J_1)*K_{2}(J_2)$, where $J=(J_{1},J_{2})$, $J_{1}$ is an $m_{1}$-tuple of positive integers and $J_{2}$ is an $m_{2}$-tuple of positive integers. Thus, applying the latter formula in the case of polytopal spheres $K_{1}=K_{P_1}$ and $K_{2}=K_{P_2}$, one obtains a combinatorial equivalence: $(P_{1}\times P_{2})(J)=P_{1}(J_{1})\times P_{2}(J_{2})$. By the equivalence proved above, (d) implies a formal $\zp$ over a graph-associahedron $P$ is a product of a number of manifolds diffeomorphic to $M_{1}$, $M_{2}$, and $M_{3}$.
Therefore, $\mathcal Z_{P(J)}$ is homeomorphic to a product of a number of manifolds homeomorphic to $\mathcal Z_{I^1(J)}$, $\mathcal Z_{P_5(J)}$, and $\mathcal Z_{P_6(J)}$. Note that $\mathcal Z_{I^1(J)}=S^{2|J|-1}$, since $I^1(J)=\Delta^{|J|-1}$. By~\cite[Theorem 2.4]{G-LdM} $\mathcal Z_{P_5(J)}$ and $\mathcal Z_{P_6(J)}$ are homeomorphic to connected sums of sphere products and thus are formal spaces. This proves the equivalence $(d)\Leftrightarrow (e)$ and finishes the proof of the whole statement.
\end{proof}

\begin{rema}
The equivalence of (a), (b), and (d) in Theorem~\ref{mainFormal} corrects the case of permutohedra in the statements of Theorem 3 and Proposition 4 in the short note~\cite{L2}. 
\end{rema}

Recall that for a simple polytope $P$ of dimension $n\geq 2$ a simple polytope $Q=\vc^{k}(P), k\geq 0$ is obtained from $P$ as a result of a sequence of vertex truncations by hyperplanes of general position w.r.t. $P$. Note that for $k\geq 1$ the combinatorial type of $Q$, in general, depends not only on that of $P$ and on the number $k$, but also on the choice of vertices that are truncated.

\begin{coro}
Suppose $\Gamma$ is a connected graph and $P=P_{\Gamma}$ is a graph-associahedron of dimension $n\geq 2$. Then the following statements are equivalent:
\begin{itemize}
\item[(a)] $\zp$ is formal;
\item[(b)] $\mathcal Z_{\vc^{k}(P)}$ is formal for any $k\geq 0$;
\item[(c)] $\mathcal Z_{\vc^{k}(P)}$ is formal for some $k\geq 0$.
\end{itemize} 
\end{coro}
\begin{proof}
To prove the equivalence of (a) and (b) it suffices to show that (a) implies (b).
By Theorem~\ref{mainFormal} formality of $\zp$ over a connected graph $\Gamma$ implies that $P$ of dimension $n\geq 2$ is either a pentagon, or a hexagon. Then $\vc^{k}(P)$ with $k\geq 0$ is an $m$-gon with $m\geq 5$ and by~\cite{mcga79} $\mathcal Z_{\vc^{k}(P)}$ is a connected sum of sphere products with two spheres in each product thus being a formal space.

To prove the equivalence of (a) and (c) it suffices to show that (c) implies (a). Assume the converse; then by Theorem~\ref{mainFormal} there exists a nontrivial triple Massey product in $H^*(\zp)$ described in Lemma~\ref{MasseyGA}. A vertex cut of a simple polytope $P$ is equivalent to a stellar subdivision of a maximal simplex of $K_P$. It follows that the induced graph in $K_{\vc(P)}$ on the same vertex set as $K_P$ coincides with the graph (i.e., 1-skeleton) of $K_P$. Observe that both triple nontrivial Massey products described in (1) and in (2) of Lemma~\ref{MasseyGA} depend only on the graph of $K_P$. Therefore, $H^*(\mathcal Z_{\vc^{k}(P)})$ has a nontrivial triple Massey product and thus $\mathcal Z_{\vc^{k}(P)}$ can not be formal (for any $k\geq 0$). We get a contradiction, which finishes the proof.
\end{proof}

\end{document}